\documentclass[11pt, reqno]{amsart}

\usepackage[utf8]{inputenc}
\usepackage[british]{babel}
\usepackage[a4paper, margin=2.5cm]{geometry}
\usepackage{amsmath, amsthm, amssymb, amsfonts}
\usepackage{mathtools}
\usepackage{thmtools}
\usepackage[shortlabels]{enumitem}
\usepackage[font={small, it}]{caption}
\usepackage{subcaption}
\usepackage{microtype}
\usepackage{xcolor}
\usepackage{bm}
\usepackage{mathabx}
\usepackage[backref]{hyperref}

\newcommand{\cA}{\ensuremath{\mathcal A}}
\newcommand{\cB}{\ensuremath{\mathcal B}}

\newcommand{\cI}{\ensuremath{\mathcal I}}

\newcommand{\cM}{\ensuremath{\mathcal M}}

\newcommand{\cS}{\ensuremath{\mathcal S}}

\newcommand{\GG}{\ensuremath{\mathbb G}}

\newcommand{\eps}{\varepsilon}
\renewcommand{\phi}{\varphi}
\renewcommand{\rho}{\varrho}

\DeclareMathOperator*{\argmax}{arg\,max}

\DeclareMathOperator*{\E}{\mathbb{E}}
\DeclareMathOperator*{\N}{\mathbb{N}}

\DeclareMathOperator*{\R}{\mathbb{R}}

\newcommand{\conv}{\ensuremath{\mathrm{conv}}}

\newcommand{\balpha}{\ensuremath{\bm{\alpha}}}

\newcommand*{\bfm}[1]{\ensuremath{\mathbf{#1}}}

\let\setminus=\smallsetminus

\newcommand{\Gnp}{G(n, \balpha, P)}

\newcommand{\osref}[2]{%
  \setlength\abovedisplayskip{5pt plus 2pt minus 2pt}
  \setlength\abovedisplayshortskip{5pt plus 2pt minus 2pt}
  \ensuremath{\overset{\text{#1}}{#2}}
}

\definecolor{mypurp}{rgb}{0.22, 0.1, 0.72}

\declaretheorem[parent=section]{theorem}
\declaretheorem[sibling=theorem]{lemma}
\declaretheorem[sibling=theorem]{proposition}
\declaretheorem[sibling=theorem]{claim}

\declaretheorem[sibling=theorem]{conjecture}

\setlist{itemsep=0.1em, topsep=0.1em, parsep=0.1em, partopsep=0.1em}

\hypersetup{
  colorlinks,
  linkcolor={red!60!black},
  citecolor={green!50!black},
  urlcolor={blue!80!black}
}

\colorlet{RoyalRed}{red!70!black}
\definecolor{RoyalBlue}{rgb}{0.25, 0.41, 0.88}
\definecolor{RoyalAzure}{rgb}{0.0, 0.22, 0.66}
\definecolor{RoyalBrown}{rgb}{0.41, 0.41, 0.1}

\newcommand{\calA}{\ensuremath{\mathcal A}}

\title{The Chromatic Number of Dense Random Block Graphs}

\author{Anders Martinsson}
\address{Anders Martinsson, Department of Computer Science, ETH Z\"{u}rich, 8092
Z\"{u}rich, Switzerland}
\email{anders.martinsson@inf.ethz.ch}

\author{Konstantinos Panagiotou}
\address{Konstantinos Panagiotou, Institute of Mathematics, University of
Munich, D-80333 Munich, Germany}
\email{kpanagio@math.lmu.de}

\author{Pascal Su}
\address{Pascal Su, Department of Computer Science, ETH Z\"{u}rich, 8092
Z\"{u}rich, Switzerland}
\email{sup@inf.ethz.ch}

\author{Milo\v{s} Truji\'c}
\address{Milo\v{s} Truji\'{c}, Department of Computer Science, ETH Z\"{u}rich,
8092 Z\"{u}rich, Switzerland}
\email{mtrujic@inf.ethz.ch}

\thanks{The research leading to these results has received funding from grant
no.\ 200021 169242 of the Swiss National Science Foundation (PS, MT) and from
the European Research Council, ERC Grant Agreement 772606–PTRCSP (KP). Part of
this work has been completed at a workshop of the research group of Angelika
Steger in Buchboden in July 2019}

\begin{document}

\begin{abstract}
  The chromatic number $\chi(G)$ of a graph $G$, that is, the smallest number of
  colors required to color the vertices of $G$ so that no two adjacent vertices
  are assigned the same color, is a classic and extensively studied parameter.
  Here we consider the case where $G$ is a random block graph, also known as the
  stochastic block model. The vertex set is partitioned into $k\in\mathbb{N}$
  parts $V_1, \dotsc, V_k$, and for each $1 \le i\le j\le k$, two vertices $u
  \in V_i, v\in V_j$ are connected by an edge with some probability $p_{ij} \in
  (0,1)$ independently. Our main result pins down the typical asymptotic value
  of $\chi(G)$ and establishes the distribution of the sizes of the color
  classes in optimal colorings. We discover that in contrast to the case of a
  binomial random graph $G(n,p)$, that corresponds to $k=1$ in our model, where
  the average size of a color class in an (almost) optimal coloring essentially
  coincides with the independence number, the block model reveals a more diverse
  picture: the ``average'' class in an optimal coloring is a convex combination
  of several types of independent sets that vary in total size as well as in the
  size of their intersection with each $V_i$, $1\le i \le k$.
\end{abstract}

\maketitle

\section{Introduction \& Main Result}

\subsection*{Chromatic Number of Random Graphs}

The chromatic number is a central and well-studied parameter in graph theory
with many applications in various other areas. For a graph $G$, the
\emph{chromatic number} $\chi(G)$ is defined as the smallest number of colors
required for coloring the vertices of $G$ such that no two adjacent vertices
are assigned the same color.

In this paper we consider the chromatic number in the case where the underlying
graph is random. The typical value and the distribution of $\chi(G(n,p))$,
where $G(n,p)$ is the binomial random graph with $n$ vertices and where each
edge is included independently with probability $p$, has been a topic of
intense study since the seminal works of Erd\H{o}s and
R\'{e}nyi~\cite{erdds1959random, erdos1960evolution} appeared. In a
breakthrough paper from 1978, Bollob\'{a}s~\cite{bollobas1988chromatic}
obtained the first asymptotically tight result: he established that for $p\in
(0,1)$, with high probability (w.h.p.), that is, with probability tending to 1
as $n \to\infty$,
\begin{equation}
\label{eq:chiGnp}
	\quad \chi(G(n,p)) = (1 + o(1)) \frac{n}{c(p) \ln n},
	\quad
	\text{where}
	\enspace
	c(p) = -\frac{2}{\ln(1-p)}.
\end{equation}
Actually, in~\cite{bollobas1988chromatic} much more was established. It has
long been known that w.h.p.\ the independence number $\alpha(G(n,p))$, the size
of the largest independent set in $G(n,p)$, equals $(1+o(1))c(p)\ln n$,
see~\cite{matula1976largest}. In addition to proving~\eqref{eq:chiGnp},
Bollob\'{a}s showed that one can color $G(n,p)$ almost optimally by covering
\emph{essentially all} of its vertices with independent sets that are roughly
of size $\alpha(G(n,p))$. More precisely, w.h.p.\ any (almost) optimal coloring
of $G(n,p)$ consists of color classes with asymptotic size $c(p)\ln n$ that
cover $n-o(n/\ln n)$ vertices.

The paper of Bollob\'{a}s initiated a long line of research concerned with
studying various properties of the distribution of the chromatic number. The
currently most accurate result on the asymptotic value of $\chi(G(n,p))$ for $p
\in (0,1)$ is due to Heckel~\cite{heckel2018chromatic}, who improved previous
results by several authors, e.g.~\cite{fountoulakis2008t, mcdiarmid1989method,
mcdiarmid1990chromatic, panagiotou2009note}, and where she showed upper and
lower bounds for $\chi(G(n,p))$ that are within $o(n/\ln^2n)$.

Apart from the probable asymptotic value of $\chi(G(n,p))$, other parameters of
it have been of considerable interest and difficulty. Most notably, the
question about the \emph{concentration} of $\chi(G(n,p))$, that is, the
smallest size of an interval in which $\chi(G(n,p))$ is located w.h.p.\ has
been a point of focus since the papers of Erd\H{o}s and R\'{e}nyi, see
also~\cite{bollobas2004sharp}. In a recent remarkable breakthrough,
Heckel~\cite{heckel2021non} and Heckel and Riordan~\cite{heckel2021does} showed
polynomial non-concentration bounds for $\chi(G(n,p))$, thus answering a
long-standing open question.

\subsection*{Stochastic Block Model}

In this paper we study the chromatic number of random graphs in the so-called
{\em stochastic block model} (also known as the \emph{planted partition model}).
The model is a generalization of $G(n,p)$ and is defined as follows. Given $k
\in \N$, let $P = (p_{ij})_{i, j \in [k]}$, where $[k] = \{1, \dots, k\}$, be a symmetric matrix with all
entries $p_{ij} \in (0, 1), i,j\in[k]$.  For brevity we sometimes write  $p_i$  for $p_{ii}$.
Moreover, let $\balpha = (\alpha_1, \dotsc,
\alpha_k)$ be a vector of (1-)norm $|\balpha| = 1$ and with all $\alpha_i \in
(0, 1]$. For an integer $n \in \N$ we let $G(n, \balpha, P)$ be a random graph
$(V, E)$ obtained as follows. The vertex set $V = V_1 \cup \dotsb \cup V_k$
consists of $k$ disjoint parts such that $|V_i| = \lfloor \alpha_i n \rfloor$ for every $i \in
[k-1]$ and $\sum_{i \in [k]} |V_i| = n$. In the seuqel we will ignore rounding issues, since they have no effect on our calculations; we thus assume  that $|V_i| = \alpha_i n, i\in[k]$. Furthermore, for $i, j \in [k]$, two
distinct vertices $u \in V_i$ and $v \in V_j$ form an edge $uv \in E$ with
probability $p_{ij}$ independently. Throughout the paper we think of $k \geq 1$
as a fixed integer, i.e.\ the number of parts $V_i$ is fixed and independent of
$n$, and $P$ as a fixed matrix, that is, we only consider (dense) graphs with
w.h.p.\ $\Omega(n^2)$ edges. We call $G \sim \Gnp$ a {\em random block graph}.
Clearly, such a graph model is a direct generalization of the
Erd\H{o}s-R\'{e}nyi binomial random graph $G(n, p)$, which is obtained by
choosing $k=1$ and $p_1 = p$.

The stochastic block model is rather flexible and it enables us to describe a
variety of situations; it is very much interconnected with the clustering
problem, where we want to partition the vertices of a given graph into
``strongly connected'' parts with ``weak'' inter-class interactions. It is thus
no surprise that it appears as a natural model in several contexts, for example
in statistics, machine learning, physics, and computer science. Its applications
range from social networks to image processing and to genetics, see
e.g.~\cite{Newman2566,Pritchard945,JM00} for some influential papers in this
context, and various properties of the model have been studied in
physics~\cite{pnassbm, holland1983stochastic} and mathematics and computer
science~\cite{boppana1987eigenvalues, bui1987graph, acopartitioning,
jersorpatitioning, recostrunction}. For further history, reference, and
discussion, we refer to the amazing survey~\cite{abbe2017community}.

Recently there have been several papers that establish generalizations of
well-known results about properties of the binomial random graph in the more
general stochastic block model. For instance,
Hamiltonicity~\cite{anastos2019hamiltonicity} and the size of the largest
independent set/clique~\cite{dolevzal2017cliques} (in a richer model that we
will also discuss). Our focus here is to determine the asymptotic value of the
chromatic number of random block graphs. As we will see shortly, a direct
consequence of the methods we employ are precise bounds on the types (sizes) of
independent sets these random graphs have.

\subsection{Main results.}

As we previously mentioned, an (almost) optimal coloring of the binomial random
graph $G(n,p)$ for $p\in (0,1)$ has typically a rather simple structure and can
be constructed \emph{greedily} in the following sense: almost all $n$ vertices
are covered by independent sets that are of nearly maximum size, that is, of
size roughly $c(p)\ln n$, where $c(p) = -2/\ln(1-p)$ is defined
in~\eqref{eq:chiGnp}. As we shall see, the structure of optimal colorings is
more intricate and diverse when we consider the broader model of random block
graphs.

In order to formulate our results we first introduce some notation. Let $k \in
\N$ and $G = G(n, \balpha, P)$ where $\balpha \in \mathbb{R}^k$ and $P \in
\mathbb{R}^{k \times k}$. Before we consider the chromatic number of $G$ we look
at the distribution of independent sets in $G$, as the two parameters are
inherently dependent on each other. Our description is heavily
inspired by the presentation in~\cite{dolevzal2017cliques}, but it is adapted
and generalized to fulfill our needs. For a vector $\bfm{c} \in \R^k$ and $I
\subseteq [k]$ define the map
\begin{equation}\label{eq:function-g}
  g(\bfm{c}, I) := \sum_{i \in I} c_i + \frac{1}{2} \sum_{i, j \in I} c_i c_j
  \ln (1 - p_{ij}).
\end{equation}
The quantity has a natural interpretation. Let $X_{\bfm{c}, I}$ be the number of
independent sets in $G$ that intersect each $V_i$, $i \in I$, at
$c_i\ln n + O(1)$ vertices. Then, as it turns out (see Section~\ref{sec:ind-sets})
\[
  \frac{1}{\ln^2n} \ln \mathbb{E}[X_{\bfm{c}, I}] = g(\bfm{c}, I) + o(1).
\]
From Markov's inequality we readily obtain that w.h.p.~$X_{\bfm{c}, I} = 0$ if
$g(\bfm{c},I) < 0$. Thus, let $\cA \subset \R^k$ be defined as
\begin{equation}\label{eq:set-A}
  \cA := \big\{ \bfm{c} \in {\R}_{\geq 0}^k : g(\bfm{c}, I) \ge 0 \text{ for all
  } \varnothing \neq I\subseteq [k] \big\}.
\end{equation}
As we will show, the (integer) vectors in $\ln n \cdot \cA$ essentially describe all \emph{admissible types} of independent sets
that we encounter in $G$ w.h.p., where ``type'' refers to a vector $\bfm{t} =
(t_1, \dotsc, t_k)$, and an independent set is of type $\bfm{t}$ if it
intersects each $V_i$, $i \in [k]$, in $t_i$ vertices (see
Figure~\ref{fig:setillu} for an illustration when $k=2$).

\begin{figure}[!htbp]
  \centering
  \includegraphics[scale=0.7]{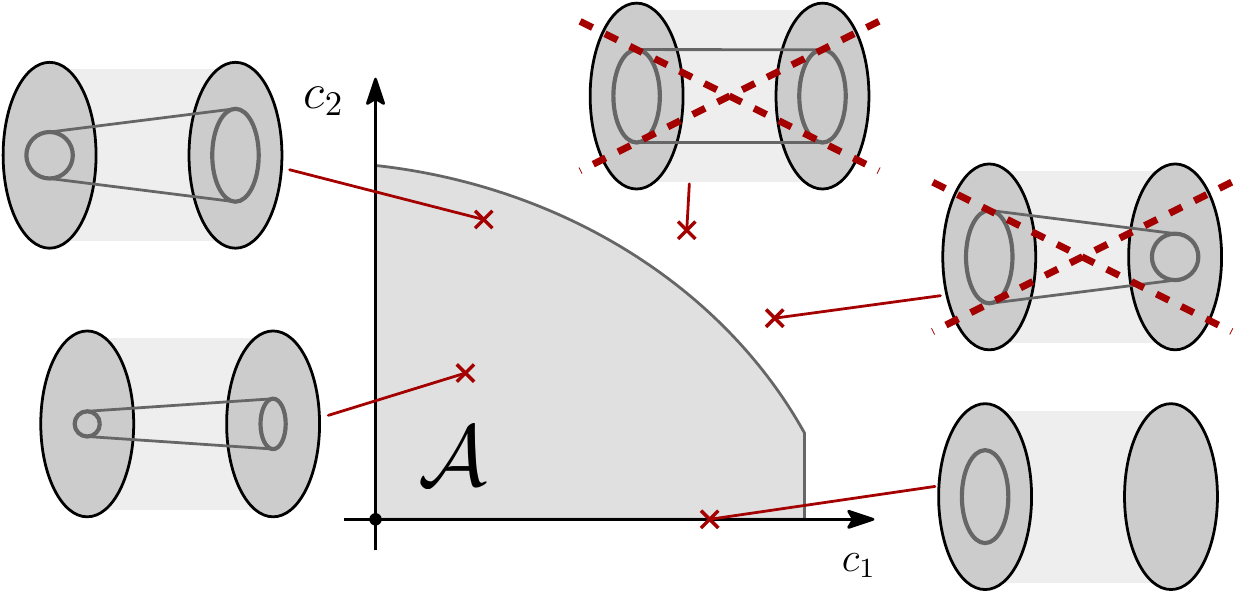}
  \caption{An example of the shape of $\calA$ and admissible types of
  independent sets in the case $k=2$. Here $p_{2} \ge p_{1} \ge p_{12}$.}
 \label{fig:setillu}
\end{figure}

Actually, we will show in Section~\ref{sec:ind-sets} even more, namely that
w.h.p.~every sufficiently large subset of $G$ contains an independent set of any
type in $(1-o(1))\ln n \cdot \cA$; this paves the way for greedily coloring $G$
with great flexibility. Our main result reads as follows, where we write conv$(\cA)$ for the convex hull of $\cA$.
\begin{restatable}{theorem}{chromaticnumber}\label{thm:main-chromatic-num}
  Let $k \geq 1$, $\balpha \in (0, 1]^k$ with $|\balpha| = 1$, and let $P =
  (p_{ij})_{i, j \in [k]}$ be a symmetric matrix with all $p_{ij} \in (0, 1)$.
  Consider a random block graph $G \sim \Gnp$. Then w.h.p.\
  \[
    \chi(G) = (1 + o(1)) \frac{n}{c^\star \ln n},
  \]
  where $c^\star$ is given as
  \begin{equation}\label{eq:c-star}
    c^\star = c^\star(\bm \alpha, P) = \max \big\{ |\bfm{c}| : \bfm{c} \in
    \conv(\cA) \cap \{ t \cdot\balpha  : t \in {\R}_{\geq 0} \} \big\}.
  \end{equation}
\end{restatable}
Let us for a moment dwell on the definition of $c^\star$. Consider an optimal
coloring of a typical instance of $G$, which is merely a partition of the
vertices of $G$ into independent sets $S_1, \dotsc, S_{\chi(G)}$. We assume that
$\chi(G) = n/(c \ln n)$ and want to determine $c$. As already mentioned, the
$S_i$'s have type  in $\ln n \cdot \cA$. Moreover, obviously
$|V_i| = \sum_{1 \leq j \leq \chi(G)} |V_i \cap S_j|$. Thus, the average intersection of
the color classes with the part $V_i$ is $ \overline{s_i} := \alpha_i n/\chi(G)
= c\alpha_i \ln n$. In conclusion, in any (in particular, in an {\em optimal})
coloring, the average intersection of the color classes with each $V_i$ is
proportional to $\ln n \cdot \balpha$ and is furthermore a {\em convex
combination} of some types in $\ln n \cdot \cA$. Hence, it comes as no surprise
that in order to determine $\chi(G)$, $c$ should be chosen to be maximal under
these side constraints, and this is exactly~\eqref{eq:c-star}.

In the proof, which is conducted in Section~\ref{sec:chi}, we {\em construct an explicit coloring of $G$} with the claimed
number of colors, that is, we carefully pick different types of independent sets
from $\ln n\cdot \cA$ and cover with them all but at most $o(n/\ln n)$ vertices.
In particular, depending very much on the shape of $\cA$, this may result in
different types of colorings: we may end up coloring all parts $V_1, \dotsc,
V_k$ independently with different colors, or at the other end of the spectrum,
we may choose just a single type $\bfm{t}$ such that $t_i / t_j \sim \alpha_i /
\alpha_j$ for all $i,j\in[k]$ and thus cover (almost) all vertices just with
sets of type $\bfm{t}$. These two cases are not exhaustive and as it turns out,
in general we may cover (almost) all vertices with independent sets of $k+1$
different types.

As a remark, by taking $k = 1$, $\alpha_1 = 1$, and $p_{1} = 1/2$, we recover
the classic result of Bollob\'{a}s~\cite{bollobas1988chromatic}. In
Section~\ref{sec:specialcases} we present various special cases of the result.
Among others, we study the case $k=2$ in detail and characterize explicitly in
all cases the structure of the optimal colorings. Moreover, for general
$k\in\mathbb{N}$ we characterize the cases in which
\[
	\chi(G) \sim \sum_{1 \le i \le k} \chi(G[V_i]),
\]
that is, an optimal coloring of $G$ is essentially obtained by coloring each of
the $k$ subgraphs individually; as we show, this happens if and only if $p_{ij}
\ge 1-\sqrt{(1-p_{i})(1-p_{j})}$. Our last example concerns one more
relevant case, namely when there is some homogeneity with respect to the edge
probabilities. In particular, we assume that all inter-class probabilities
$p_{ij}$, for $i \neq j$, are equal, and all intra-class probabilities are also
equal. In that case, we determine explicitly the asymptotic value of the
chromatic number.

Note that we determine $\chi(G)$ in the case when $p_{ij}$'s and $\alpha_i$'s
are fixed and independent of $n$. Extending this for $p_{ij} = p_{ij}(n)$ and
$\alpha_i = \alpha_i(n)$ remains an open problem for further research.

\subsection{Graph limits}

In an even more general setting we may look at limits of dense graph sequences.
For a detailed introduction to the topic we refer the reader to the wonderful
book of Lov\'{a}sz~\cite{lovasz2012large}. A {\em graphon} is a symmetric
measurable function $W \colon \Omega \times \Omega \to [0, 1]$, where $\Omega$
is a probability space. In order to show that every graphon $W$ is attained as a
limit of a sequence of finite graphs, Lov\'{a}sz and
Szegedy~\cite{lovasz2007szemeredi} introduced a random graph model $\GG(n, W)$
defined as follows. The vertex set of a graph $G \sim \GG(n, W)$ is $[n]$.  In
order to sample a graph $G \sim \GG(n, W)$, one first generates a sequence of
$n$ points $x_1, \dotsc, x_n \in \Omega$ and subsequently makes $x_ix_j$ an edge
with probability $W(x_i, x_j)$ independently of everything else. Clearly, the
class of stochastic block models $\Gnp$ is (essentially) a special case of
$\GG(n, W)$.

As in the block model, we define a continuous version of the function $g$
in~\eqref{eq:function-g}, where we replace sums by integrals and values by
densities. More specifically (and compare also with \cite{dolevzal2017cliques}),
define for a measurable $I \subseteq \Omega$ and a non-negative $L^1$-function
$c$,
\[
  \tilde g(c, I) = \int_{x \in I} c(x) d\nu + \frac{1}{2} \int_{(x, y) \in I
  \times I} c(x)c(y) \ln(1 - W(x, y)) d\nu^2.
\]
Moreover, the set of ``admissible types'' is defined analogously by
\[
  \tilde\cA = \{ c : c \text{ is a non-negative $L^1$-function on $\Omega$ such
  that } \tilde g(c, I) \geq 0 \text{ for all measurable } I \subseteq \Omega \}.
\]
We suspect that the chromatic number of $G \sim \GG(n, W)$ is then obtained in a
way similar to the one from Theorem~\ref{thm:main-chromatic-num}.
\begin{conjecture}
  Let $W \colon \Omega \times \Omega \to [0, 1]$ be a graphon with
  essential infimum in $(0,1)$ and consider a graph $G \sim \GG(n, W)$. Then w.h.p.\
  \[
    \chi(G) = (1 + o(1)) \frac{n}{c^\star \ln n},
  \]
  where $c^\star$ is given as
  \[
    c^\star = c^\star(W) = \sup{\big\{ \lVert c \rVert_1 : c \in
    \mathrm{conv}(\tilde\cA) \cap \{ t \cdot \nu : t \in {\R}_{\geq 0} \}
    \big\}}.
  \]
\end{conjecture}
However, this seemed to be out of reach for our techniques and we leave it as a
question for further research.

\section{Independent Sets}\label{sec:ind-sets}

In this section we study the distribution of independent sets in the random
block graph $\Gnp$. This serves as a main ingredient towards deriving the
desired bounds on the chromatic number later on. Throughout, for $\bfm{t} \in
\N_0^k$ we say that a set $S \subseteq V(G)$ is a $\bfm{t}$-{\em set} or
\emph{of type \bfm{t}} in $\Gnp$ if $S \cap V_i = t_i$ for every $i \in [k]$.
Vectors are denoted by lower-case bold letters. Given vectors $\bfm{u}, \bfm{v}
\in \N_{\geq 0}^k$, we write $\bfm{u} \leq \bfm{v}$ if $u_i \leq v_i$ for all $i
\in [k]$.

Our starting point and main technical tool in this section is a simple
consequence of Janson's inequality,
see~\cite[Section~21.6]{frieze2016introduction}, which we restate in a variant
convenient for our application.
\begin{theorem}[Janson's inequality]\label{thm:janson}
  Let $k \in \N$, $\balpha \in (0, 1]^k$ with $|\balpha| = 1$, and let $P =
  (p_{ij})_{i, j \in [k]}$ be a symmetric matrix with all $p_{ij} \in (0, 1)$.
  Consider a family $\{ S_i \}_{i \in \cI}$ of subsets of the vertex set $[n]$
  and let $G \sim \Gnp$. For each $i \in \cI$, let $X_i$ denote the indicator
  random variable for the event $\{S_i \text{ is an independent set in } G\}$
  and, for each ordered pair $(i, j) \in \cI \times \cI$, write $X_i \sim X_j$
  if the variables $X_i$ and $X_j$ are not independent. Let
  \begin{align*}
    X := \sum_{i \in \cI} X_i, \qquad \mu := \E[X], \qquad \text{and} \qquad
    \overline\Delta := \sum_{\substack{(i, j) \in \cI \times \cI \\ X_i \sim X_j}} \E[X_i
    X_j].
  \end{align*}
  Then
  \[
    \Pr[X = 0] \leq e^{-\mu^2/(2\overline\Delta)}.
  \]
\end{theorem}
The next lemma is the central result of this section. In simple terms, it states
that w.h.p.\ whenever we take a sufficiently large subset of vertices of $G \sim
\Gnp$, there is an independent $\bfm{t}$-set, for any $\bfm{t}$ that ``falls''
within the set $\cA$, that is, $\bfm{t} \in (1 - o(1)) \ln n \cdot \cA$. This
lemma alone allows us to greedily take out independent sets (color classes) from
$\Gnp$ as long as there is some ``large'' set of vertices remaining in each
$V_i$.

\begin{lemma}\label{lem:non-existance-bound}
  Let $G \sim \Gnp$. Let $\bfm{s} \in \N_0^k$ be such that $s_i \geq \alpha_i
  n/\ln^2 n$ for all $i \in [k]$ and $S \subseteq [n]$ be an $\bfm{s}$-set. Then, for every $\bfm{t} \in (\ln n - 7\ln\ln n) \cdot \cA \cap \N_0^k$ and
  $X_\bfm{t}$ being the random variable denoting the number of independent
  $\bfm{t}$-sets in $G[S]$
  \[
    \Pr[X_\bfm{t} = 0] = e^{- \Omega(n^2/\ln^{8} n)}.
  \]
\end{lemma}
\begin{proof}
  Let us write $X_\bfm{t} = \sum_{I \in \cS} X_I$, where $\cS$ is the
  family of all subsets of $S$ that intersect each $V_i$, $i \in [k]$, in
  exactly $t_i$ vertices, and $X_I$ is an indicator random variable for the
  event that $I$ is an independent set in $G[S]$. Set
  \[
    \mu := \E[X_\bfm{t}] \qquad \text{and} \qquad \overline\Delta := \sum_{\substack{(I, J) \in \cS \times \cS \\ X_I \sim X_J}} \E[X_I X_J].
  \]
  This puts us directly into the setup of Janson's inequality
  (Theorem~\ref{thm:janson}) with the goal to show
  \[
    \Pr[X_\bfm{t} = 0] \leq e^{-\mu^2/(2\overline\Delta)} \overset{!}{=}
    e^{-\Omega(n^2/\ln^{8} n)}.
  \]
  The whole proof boils down to showing that the $\overline\Delta$ term can be
  bounded by
  \begin{equation}\label{eq:delta-bound}
    \overline\Delta = O \Big( \mu^2 \cdot \frac{\ln^{8} n}{n^2} \Big),
  \end{equation}
  which is what we accomplish in the remainder. First, it is convenient to determine $\mu = \E[X_\bfm{t}]$ as it helps
  simplify some calculations. For each $i \in [k]$, there are $\binom{s_i}{t_i}$
  choices for the intersection of a $\bfm{t}$-set with $S \cap V_i$.
  Additionally, in order for such a $\bfm{t}$-set to be an independent set in
  $G[S]$, none of the $\binom{t_i}{2}$ pairs can form an edge, which happens
  with probability $(1 - p_i)^{\binom{t_i}{2}}$. Lastly, no two vertices $u, v$
  in the $\bfm{t}$-set with $u \in V_i$ and $v \in V_j$ can form an edge in
  $G[S]$, which happens with probability $(1 - p_{ij})^{t_i t_j}$. Putting all
  of this together, we directly get
  \begin{equation}\label{eq:expectation}
    \mu = \prod_{1 \leq i \leq k} \binom{s_i}{t_i}\cdot \prod_{1 \leq i \leq k}
    (1 - p_i)^{\binom{t_i}{2}} \cdot \prod_{1 \leq i < j \leq k} (1 -
    p_{ij})^{t_it_j}.
  \end{equation}
  We now turn our attention in bounding the $\overline\Delta$ term as promised.
  Note that the $\overline\Delta$ term depends only on those sets which have at least one edge in common, that is,  they
  ``overlap'' in at least two vertices. We denote the overlap vector by
  $\bfm{o}$ and note that $\bfm{o} \leq \bfm{t}$ and $|\bfm{o}| \geq 2$. Each
  $o_i$, for $i \in [k]$, measures the ``overlap'' of the sets inside of the
  part $V_i$. For a fixed overlap vector $\bfm{o}$ and a fixed $i \in [k]$,
  there are at most
  \[
    \binom{s_i}{t_i} \binom{t_i}{o_i} \binom{s_i - t_i}{t_i - o_i}
  \]
  choices for two $\bfm{t}$-sets which intersect in exactly $o_i$ vertices
  within $V_i$. Similarly as above when deriving the expectation, the
  probability of both such $\bfm{t}$-sets being independent is given by a term
  for intra-class edges and inter-class edges and is exactly
  \[
    (1 - p_i)^{2\binom{t_i}{2} - \binom{o_i}{2}} \cdot \prod_{j \neq i} (1 -
    p_{ij})^{2t_it_j - o_io_j}.
  \]
  Thus, the contribution to the $\overline\Delta$ term of a fixed overlap vector
  $\bfm{o}$ is given by
  \[
    \prod_{1 \leq i \leq k} \binom{s_i}{t_i} \binom{t_i}{o_i}
    \binom{s_i-t_i}{t_i-o_i} \cdot \prod_{1 \leq i \leq k} (1-p_i)^{2
    \binom{t_i}{2}-\binom{o_i}{2}} \cdot \prod_{1 \leq i < j \leq k}
    (1-p_{ij})^{2 t_it_j-o_io_j}.
  \]
  Then, by summing up over all choices of $\bfm{o}$, we get
  \begin{align*}
    \overline\Delta & = \sum_{\substack{\bfm{o} \leq \bfm{t} \\
    |\bfm{o}| \geq 2}} \Big( \prod_{1 \leq i \leq k} \binom{s_i}{t_i} \binom{t_i}{o_i}
    \binom{s_i-t_i}{t_i-o_i} \cdot \prod_{1 \leq i \leq k}
    (1-p_i)^{2 \binom{t_i}{2}-\binom{o_i}{2}} \cdot \prod_{1 \leq i < j \leq k}
    (1-p_{ij})^{2 t_it_j-o_io_j} \Big) \\
    & \osref{\eqref{eq:expectation}}= \mu \cdot \sum_{\substack{\bfm{o} \leq \bfm{t} \\
    |\bfm{o}| \geq 2}} \Big( \prod_{1 \leq i \leq k} \binom{t_i}{o_i}
    \binom{s_i-t_i}{t_i-o_i} \cdot \prod_{1 \leq i \leq k}
    (1-p_i)^{\binom{t_i}{2}-\binom{o_i}{2}} \cdot \prod_{1 \leq i < j \leq k}
    (1-p_{ij})^{t_it_j-o_io_j} \Big) \\
    & \osref{\eqref{eq:expectation}}= \mu^2 \cdot \sum_{\substack{\bfm{o} \leq
    \bfm{t} \\ |\bfm{o}| \geq 2}} \Big( \underbrace{\prod_{1 \leq i \leq k}
    \frac{\binom{t_i}{o_i} \binom{s_i - t_i}{t_i - o_i}}{\binom{s_i}{t_i}} \cdot
    \prod_{1 \leq i \leq k} (1 - p_i)^{-\binom{o_i}{2}} \cdot \prod_{1 \leq i <
    j \leq k} (1 - p_{ij})^{-o_io_j}}_{:= f(\bfm{o})} \Big) \\
    & = \mu^2 \cdot \sum_{\substack{\bfm{o} \leq \bfm{t} \\ |\bfm{o}| \geq 2}}
    f(\bfm{o}).
  \end{align*}
  To complete the proof we aim to give a bound of the order $\ln^{8} n/n^2$ for the sum in the last expression above. To this end, we first show that  the whole sum is essentially dominated by those $f(\bfm{o})$ with $|\bfm{o}| = 2$.  Consider an arbitrary $\bfm{o}$ and let $\bfm{e}_z \in \{0, 1\}^k$ be the unit vector with $(\bfm{e}_z)_z = 1$ for  $z \in [k]$.  We derive
  \[
    \frac{f(\bfm{o})}{f(\bfm{o} + \bfm{e}_z)} \geq \frac{\binom{t_z}{o_z}
    \binom{s_z - t_z}{t_z - o_z}}{\binom{t_z}{o_z + 1} \binom{s_z - t_z}{t_z -
    (o_z + 1)} } \cdot (1 - p_z)^{\binom{o_z + 1}{2} - \binom{o_z}{2}} \cdot
    \prod_{j \neq z} (1 - p_{zj})^{(o_z + 1)o_j - o_zo_j}.
  \]
  Using the fact that $t_i \leq c(p_i)\ln n$ (by definition \eqref{eq:set-A} of
<  $\cA$) and $s_i \geq \alpha_i n/\ln^2 n$, this can be simplified (by standard
  manipulations of binomial coefficients $\binom{a}{b} = \binom{a}{b+1} \frac{b+1}{a-b}$ and $\binom{a+1}{2} = \binom{a}{2} + a$) to
  \begin{equation}\label{eq:ratio-between-deltas}
  \begin{split}
    \frac{f(\bfm{o})}{f(\bfm{o} + \bfm{e}_z)}
    & \geq
    \frac{ (o_z+1) (s_z -2t_z + o_z + 1)}{(t_z - o_z) (t_z-o_z)}
    \cdot (1 - p_z)^{o_z} \cdot \prod_{j \neq z} (1 - p_{zj})^{o_j} \\ & \geq
    \frac{\delta n}{\ln^4 n}
    \cdot (1 - p_z)^{o_z} \cdot \prod_{j \neq z} (1 - p_{zj})^{o_j},
  \end{split}
  \end{equation}
  for some constant $\delta > 0$ which depends only on $\alpha_i$'s and $p_i$'s.
  Let $\bfm{\tilde o} \leq \bfm{t}$ be such that $|\bfm{\tilde o}| = 2$ and
  $\bfm{\tilde o} = \bfm{e}_x + \bfm{e}_y$, for some (not necessarily distinct)
  $x, y \in [k]$. Then from \eqref{eq:ratio-between-deltas}, for every
  $\bfm{\tilde o} \leq \bfm{o} \leq \bfm{t}$ with $|\bfm{o}| \geq 3$, we obtain
  \[
    \frac{f(\bfm{\tilde o})}{f(\bfm{o})} \geq \Big( \frac{\delta n}{\ln^4 n}
    \Big)^{|\bfm{o}| - 2} \cdot \prod_{1 \leq i \leq k} (1 -
    p_i)^{\binom{o_i}{2}} \cdot \prod_{1 \leq i < j \leq k} (1 - p_{ij})^{o_i
    o_j}.
  \]
  This can be further bounded from below by
  \begin{multline}\label{eq:bound-on-quotients}
    \exp \Big( (|\bfm{o}| - 2) (\ln\delta - 4\ln\ln n) - 2\ln n \\
    + \underbrace{\sum_{1 \leq i \leq k} o_i \ln n + \frac{1}{2} \sum_{1 \leq i
    \leq k} o_i^2 \ln(1 - p_i) + \sum_{1 \leq i < j \leq k} o_io_j \ln(1 -
    p_{ij})}_{:= h(\bfm{o})} \Big).
  \end{multline}
  Let $\bfm{d} = \bfm{o} - \tilde{\bfm{o}}$ and $\eps = 7\ln\ln n/\ln n$.
  Since $\bfm{d} \leq \bfm{o} \leq \bfm{t}$ and $\bfm{t} \in (1 - \eps) \ln n
  \cdot \cA$, using the definition of $\cA$, we get
  \[
    \sum_{1 \leq i \leq k} \frac{d_i}{(1 - \eps) \ln n} + \frac{1}{2} \sum_{1
    \leq i \leq k} \frac{d_i^2}{(1 - \eps)^2 \ln^2 n} \ln(1 - p_i) + \sum_{1
    \leq i < j \leq k} \frac{d_i d_j}{(1 - \eps)^2 \ln^2 n} \ln(1 - p_{ij}) \geq
    0.
  \]
  Multiplying the whole inequality by $(1 - \eps)^2 \ln^2 n$ gives
  \begin{equation}\label{eq:lower-bound-on-h}
    h(\bfm{d}) := \sum_{1 \leq i \leq k} d_i \ln n + \frac{1}{2}\sum_{1 \leq i
    \leq k} d_i^2 \ln(1 - p_i) + \sum_{1 \leq i < j \leq k} d_i d_j \ln (1 -
    p_{ij}) \geq \eps |\bfm{d}| \ln n.
  \end{equation}
  On the other hand, since $\bfm{d} = \bfm{o} - \bfm{\tilde o}$ and $\bfm{\tilde
  o} = \bfm{e}_x + \bfm{e}_y$, we have
  \begin{multline*}
    h(\bfm{o}) - 2 \ln n \ge h(\bfm{d}) - 2\ln n + (|e_x| + |e_y|)\ln n +
    \frac{1}{2} \big( 2o_x \ln(1 - p_x) + 2o_y \ln(1 - p_y) \big) \\
    + \sum_{j \neq x} e_x o_j \ln(1 - p_{xj}) + \sum_{j \neq y} e_y o_j \ln(1
    - p_{yj}).
  \end{multline*}
  Therefore, $h(\bfm{o}) - 2\ln n \geq h(\bfm{d}) - O( |\bfm{o}|)$. By plugging
  in \eqref{eq:lower-bound-on-h} into \eqref{eq:bound-on-quotients}, and as
  $|\bfm{o}| \geq 3$, we get
  \[
  \begin{split}
    \frac{f(\bfm{\tilde o})}{f(\bfm{o})}
    & \geq \exp \Big( (|\bfm{o}| - 2)
    (\ln\delta - 4\ln\ln n) + 7(|\bfm{o}| - 2) \ln\ln n - O(|\bfm{o}|) \Big) \\
    &\geq \ln^{3(|\bfm{o}|-2)} n + o( \ln^{|\bfm{o}|-2} n) .
  \end{split}
  \]
  Clearly, by the fact that $\bfm{t} \in (1 - \eps)\ln n \cdot \cA$ and as
  $\cA$ is bounded, the norm of $\bfm{o}$ is at most $C\ln n$ for some (large)
  constant $C > 0$ depending only on $\alpha_i$'s and $p_i$'s; we may choose, for example, $C = \sum_{1 \leq i \leq k} c(p_{i})$. This finally implies
  \[
    \sum_{\substack{\bfm{o} \le \bfm{t} \\ |\bfm{o}| \ge 2}} f(\bfm{o})
    = O\Big(\sum_{2 \le i \le |\bfm{t}|} k^{i-2} \ln^{-3(i-2)} n \sum_{|\bfm{\tilde o}|=2}
    f( \bfm{\tilde o}) \Big) = O \Big( \sum_{ |\bfm{\tilde o}| = 2} f(\bfm{\tilde
    o}) \Big).
  \]
  Hence, it remains to show that $f(\bfm{\tilde o}) = O(\ln^8 n/n^2)$ when $|\bfm{\tilde o}| = 2$. Note that each such $\bfm{\tilde o}$ can be written  as $\bfm{\tilde o} = \bfm{e}_x + \bfm{e}_y$ for some $x, y \in [k]$. In the case $x \neq y$  we obtain
  \[
    f(\bfm{\tilde o}) \le t_x \frac{\binom{s_x - t_x}{t_x - 1}}{\binom{s_x}{t_x}}
    \cdot t_y \frac{\binom{s_y - t_y}{t_y - 1}}{\binom{s_y}{t_y}} \cdot (1 -
    p_{xy})^{-1}.
  \]
  Simple manipulations with binomial coefficients give
  \[
    f(\bfm{\tilde o}) \leq \frac{(t_x t_y)^2}{s_x s_y (1 - p_{xy})}.
  \]
  On the other hand, if $x = y$ (and so $\tilde o_x = 2$) we obtain
  \[
    f(\bfm{\tilde o}) = \binom{t_x}{2} \frac{\binom{s_x - t_x}{t_x -
    2}}{\binom{s_x}{t_x}} \cdot (1 - p_x)^{-1} \leq \frac{t_x^4}{s_x^2 (1 -
    p_x)}.
  \]
  Recalling that $s_i \geq \alpha_i n/\ln^2 n$ and $t_i = O(\ln n)$ shows the
  desired bound and completes the proof of the lemma.
\end{proof}

The next lemma establishes the fact that $\Gnp$ w.h.p.\ contains no independent
$\bfm{t}$-sets which lie ``outside'' of $\ln n \cdot \cA$. We start by making a useful
observation about $\cA$.
\begin{claim}\label{claim:smallvecsinA}
  There exists a constant $C=C(P)>0$ such that any vector $\bfm{c} \in
  {\R}_{\geq 0}^k$ with $|\bfm{c}|\leq C$ is contained in $\cA$.
\end{claim}
\begin{proof}
  For any $I\subseteq [k]$, we have
  \[
    g(\bfm{c}, I) \geq \sum_{i\in I} c_i + \frac12 \sum_{i', j'} c_{i'} c_{j'}
    \ln(1-\max_{i,j} p_{ij}) \geq |\bfm{c}| + \frac12 |\bfm{c}|^2
    \ln(1-\max_{i,j} p_{ij}).
  \]
  It follows that $g(\bfm{c},I) \geq 0$ for all $I\subseteq [k]$ if $|\bfm{c}|
  \leq C:=-2/\ln(1-\max_{i,j} p_{ij})$.
\end{proof}

\begin{lemma}\label{lem:no-ind-set-out-of-A}
  Let $G \sim \Gnp$. Let $X$ be the random variable denoting the number of
  independent $\bfm{t}$-sets in $G$ with $\bfm{t} \in \N_0^k \setminus (\ln n
  \cdot \cA)$. Then
  \[
    \Pr[X > 0] \leq O \big( n^{-1} \big).
  \]
\end{lemma}
\begin{proof}
  Fix any $\bfm{t} \notin \N_0^k \setminus (\ln n \cdot \cA)$ and let
  $X_{\bfm{t}}$ count the number of independent $\bfm{t}$-sets for that fixed
  $\bfm{t}$. Observe that, by the definition \eqref{eq:set-A} of $\cA$, there is
  a $I \subseteq [k]$ such that $g \big( (t_i/\ln n)_{i \in I}, I \big) < 0$.
  Let $\bfm{\tilde t} = (\tilde t_1, \dotsc, \tilde t_k)$ be defined as
  \[
    \tilde t_i =
    \begin{cases}
      t_i, & \text{ if $i \in I$}, \\
      0, & \text{ otherwise},
    \end{cases}
  \]
  and let $X_{\bfm{\tilde t}}$ be the random variable denoting the number of
  independent $\bfm{\tilde t}$-sets in $G$. Then
  \[
    \E[X_{\bfm{\tilde t}}] = \prod_{i \in I} \binom{\alpha_i n}{t_i} \cdot
    \prod_{i \in I} (1 - p_i)^{\binom{t_i}{2}} \cdot \prod_{\substack{i, j \in I
    \\ i \neq j}} (1 - p_{ij})^{t_i t_j}.
  \]
  Since $\alpha_i < 1$, using that $\binom{n}{k} \le (\frac{en}{k})^k$, this can
  further be bounded by
  \[
    \E[X_{\bfm{\tilde t}}] \leq \exp \Big( \sum_{i \in I} t_i + \sum_{i \in I}
    t_i \ln n - \sum_{i \in I} t_i \ln t_i + \frac{1}{2}\sum_{i \in I} t_i^2
    \ln(1 - p_i) + \sum_{\substack{i, j \in I \\ i \neq j}} t_i t_j \ln(1 -
    p_{ij}) \Big).
  \]
  Recall, $g \big( (t_i/\ln n)_{i \in I}, I \big) < 0$, and thus
  \[
    \sum_{i \in I} \frac{t_i}{\ln n} + \frac{1}{2} \sum_{i \in I}
    \frac{t_i^2}{\ln^2 n} \ln(1 - p_i) + \sum_{\substack{i, j \in I \\ i \neq
    j}} \frac{t_i t_j}{\ln^2 n} \ln(1 - p_{ij}) < 0.
  \]
  which yields
  \[
    \frac{1}{2} \sum_{i \in I} t_i^2 \ln(1 - p_i) + \sum_{\substack{i, j \in I
    \\ i \neq j}} t_i t_j \ln(1 - p_{ij}) < - \sum_{i \in I} t_i \ln n.
  \]
  By Jensen's inequality and the fact that $|\bfm{\tilde t}| = \Omega(\ln n)$ by
  Claim~\ref{claim:smallvecsinA}, it further follows that
  \[
    \E[X_{\bfm{\tilde t}}] \leq \exp \big( -|\bfm{\tilde t}|\ln(|\bfm{\tilde
    t}|/k) + |\bfm{\tilde t}| \big) \leq O(n^{-2}).
  \]
  Hence, by Markov's inequality $\Pr[X_{\bfm{\tilde t}} > 0] \leq
  \E[X_{\bfm{\tilde t}}] \leq O(n^{-2})$.

  Let $\mathcal{M}$ be the set of minimal $\bfm{t}$'s in $\N_0^k \setminus
  (\ln n \cdot \cA)$. Then if there is any independent set in $\N_0^k \setminus
  (\ln n \cdot \cA) $ it must also contain an $\bfm{m}$-set as a subset, for
  some $\bfm{m} \in \mathcal{M}$. Similarly as in the proof of
  Lemma~\ref{lem:non-existance-bound} (as $\cA$ is bounded), we know $|\bfm{m}|
  \le C \ln n$ for some $C > 0$ depending only on $\alpha_i$'s and $p_i$'s.
  Therefore, $|\mathcal{M}| \le (C\ln n)^k$ and by a union bound over all vectors
  in $\cM$, we get
  \[
    \Pr[X > 0] \leq \sum_{\bfm{m} \in \cM} \Pr[X_\bfm{m} > 0 \big] \leq O(n^{-1}).
  \]
  In particular, w.h.p.\ there is also no independent $\bfm{t}$-set for any
  $\bfm{t} \in {\N}_0^k \setminus (\ln n \cdot \cA)$, which completes the
  proof.
\end{proof}

\section{Chromatic Number}\label{sec:chi}

In this section we provide the proof of our main theorem. Recall, the goal is to
give a precise (up to lower order terms) bound on the chromatic number of a
random block graph $\Gnp$. For the convenience of the reader we restate our main
result.

\chromaticnumber*

\subsection{Upper bound}

Set $\eps = 7\ln\ln n/\ln n$ and let $\bfm{c^\star} = \argmax\big\{ |\bfm{c}|
: \bfm{c} \in \conv(\cA) \cap \{\balpha t : t \in \R_{\geq 0}\} \big\}$. While
constructing a coloring to give an upper bound on $\chi(G)$ we need to
distinguish two cases: $\bfm{c^\star} \in \cA$ and $\bfm{c^\star} \notin \cA$.

Assume the former, that is, $\bfm{c^\star} \in \cA$ and let $\bfm{t} = (1 - \eps)
\ln n \cdot \bfm{c^\star}$. The probability that there is a set $S \subseteq
V(G)$ with $s_i := |S \cap V_i| = \alpha_i n/\ln^2 n$ for all $i \in [k]$ and
such that $G[S]$ does not contain an independent $\bfm{t}$-set is, by
Lemma~\ref{lem:non-existance-bound} and a union bound over all choices for $S$,
at most
\[
  \prod_{1 \leq i \leq k} \binom{\alpha_i n}{s_i} e^{-\Omega(n^2/\ln^{8} n)}
  \leq e^{\sum_{i = 1}^{k} s_i \ln n} \cdot e^{-\Omega(n^2/\ln^{8} n)} =
  o(1).
\]
As a direct consequence, w.h.p.\ as long as there are at least $\alpha_i
n/\ln^2 n$ vertices remaining in every $V_i$, there is an independent
$\bfm{t}$-set where $\bfm{t} = (1 - \eps) \ln n \cdot \bfm{c^\star}$. We
construct the coloring in the usual way: repeatedly take out an independent
$\bfm{t}$-set and assign all the vertices in it a new color. By the argument
above, this is possible until there are at most $\alpha_i n/\ln^2 n$ vertices left in every set $V_i$ and $O(n/\ln^2 n)$ vertices in total, which happens simultaneously since the vector $\bfm{c}^\star$ is proportional to $\balpha$. At this point we assign to each of those uncolored vertices a color different from
all the previously used ones. Therefore, the total number of colors used is at
most
\[
  \frac{n}{|\bfm{t}|} + O\Big(\frac{n}{\ln^2 n}\Big) = \frac{n}{|\bfm{c^\star}|(\ln n -
  7\ln\ln n)} + O\Big(\frac{n}{\ln^2 n}\Big) = (1 + o(1)) \frac{n}{c^\star \ln n},
\]
as claimed.

On the other hand, if $\bfm{c^\star} \notin \cA$, w.h.p.\ no independent
$(\bfm{c^\star} \ln n)$-set exists in $G$ by
Lemma~\ref{lem:no-ind-set-out-of-A}. In order to circumvent this, we represent
$\bfm{c^\star}$ as a convex combination
\begin{equation}\label{eq:c-star-conv-comb}
  \bfm{c^\star} = \sum_{1 \leq i \leq k + 1} \lambda_i \bfm{t_i},
\end{equation}
where $\bfm{t_i} \in \cA$, $\lambda_i \in [0, 1]$ for all $i \in [k + 1]$, and
$\sum_{1 \leq i \leq k + 1} \lambda_i = 1$.

We then construct a coloring of $G$ as follows. Greedily and sequentially
select $\lambda_i n/(c^\star \ln n)$ independent $(\bfm{t_i} (1 - \eps)\ln
n)$-sets and assign all the vertices in it a new color. Note that throughout this greedy process, the number of uncolored vertices in every $V_i$ is at least
\begin{align*}
  \alpha_i n - \sum_{1 \leq j \leq k + 1} \frac{\lambda_j n}{c^\star \ln n}
  \cdot (\bfm{t_j})_i \cdot (1 - \eps) \ln n & = \alpha_i n - \frac{(1 -
  \eps)n}{c^\star} \cdot \sum_{1 \leq j \leq k + 1} \lambda_j (\bfm{t_j})_i \\
  & \osref{\eqref{eq:c-star-conv-comb}}= \alpha_i n - \frac{(1 - \eps)
  n}{c^\star} \cdot (\bfm{c^\star})_i = \frac{7 \alpha_i n \ln\ln n}{\ln n},
\end{align*}
where the last equality follows from the fact that $\bfm{c^\star} = c^\star
\cdot \balpha$ and our choice of $\eps$. This is enough for Lemma~\ref{lem:non-existance-bound} to apply and we w.h.p.\ find a next independent $(\bfm{t_i}(1-\eps)\ln n)$-set.

Let $Q_i$ be the set of uncolored vertices in every $V_i$. Since $G[V_i]$ is
distributed as $G(\alpha_i n, p_i)$ and $Q_i$ is a subset of $V_i$ of size
$|Q_i| \geq \eps\alpha_i n$ w.h.p.\ by Lemma~\ref{lem:non-existance-bound} it
follows that as long as there are at least $\alpha_i n/\ln^2 n$ vertices
remaining, generously rounding for $\balpha$, we find an independent set of size
at least $c(p_i)\ln n/2$ in $G[Q_i]$. We greedily take such sets one by one and
assign all the vertices in each a new color. Lastly, assign every uncolored
vertex a new color which was previously unused. Therefore,
\[
  \chi(G[Q_i]) \leq \frac{14 \alpha_i n \ln\ln n}{c(p_i) \ln^2 n} +
  \frac{n}{\ln^2 n} = o \Big( \frac{n}{\ln n} \Big).
\]
Consequently, the number of different colors used for the whole graph $G$ is at
most
\[
  \sum_{1 \leq i \leq k + 1} \frac{\lambda_i n}{c^\star \ln n} + \sum_{1 \leq i
  \leq k} \chi(G[Q_i]) = \frac{n}{c^\star \ln n} + o \Big( \frac{n}{\ln n}
  \Big),
\]
as $\sum_{1 \leq i \leq k + 1} \lambda_i = 1$. This confirms the claimed upper
bound.

\subsection{Lower bound}

Set $N = \chi(G)$ and consider any coloring with color classes $S_1, \dotsc,
S_N$. Trivially, for every $j \in [k]$ we have $\sum_{1 \leq i \leq N} |S_i \cap
V_j| = \alpha_j n$. So by Lemma~\ref{lem:no-ind-set-out-of-A} we may assume
every color class is an independent $\bfm{t}$-set for some $\bfm{t} \in \ln n
\cdot \cA$. Let $\overline{\bfm{t}} = \frac{1}{N} \sum_{i \in [N]} \bfm{t}_i$ and
note that $\overline{\bfm{t}} \in \{t \cdot \balpha\}$ for some $t \in \R_{\geq 0}$
and $\overline{\bfm{t}} \in \conv(\cA)$. Therefore,
\[
  N = \frac{n}{|\overline{\bfm{t}}| \ln n} \geq \frac{n}{c^\star \ln n},
\]
by maximality of $c^\star$, see \eqref{eq:c-star}. \qed

\section{Special Cases}\label{sec:specialcases}

\subsection{Two-block case}\label{sec:twoblock}

Throughout this subsection we assume that $k = 2$ and try to in detail describe
the possible structure of the set $\cA$ and independent sets of $\Gnp$. Recall,
the set $\cA$ is defined in order to describe ``feasible'' sizes of independent
sets the graph $\Gnp$ can have and in the case $k = 2$ is given as follows:
\begin{align*}
  \cA = \big\{ & 0 \leq c_1 \leq c(p_1), 0 \leq c_2 \leq c(p_2), \\
  & c_1 + c_2 + \frac{c_1^2}{2} \ln(1 - p_1) + \frac{c_2^2}{2} \ln(1 - p_2) +
  c_1c_2 \ln(1 - p_{12}) \geq 0 \big\}.
\end{align*}
The first two equations determine the size of the largest independent set inside
each of the parts $V_1$ and $V_2$ on their own by treating them as $G(\alpha_1
n, p_1)$ and $G(\alpha_2 n, p_2)$, respectively. The third inequality is what
determines the shape of $\cA$.

In particular, having $p_1$ and $p_2$ fixed, the shape of $\cA$ varies
significantly depending on $p_{12}$. Note that, by using \eqref{eq:chiGnp}, the
inequality
\[
  c_1 + c_2 + \frac{c_1^2}{2} \ln(1 - p_1) + \frac{c_2^2}{2} \ln(1 - p_2) +
  c_1c_2 \ln(1 - p_{12}) \geq 0
\]
can be rewritten as
\begin{equation}\label{eq:kequaltwo}
  (c_1 + c_2) \Big( 1 - \frac{c_1}{c(p_1)} - \frac{c_2}{c(p_2)} \Big) +
  c_1c_2 \Big( \ln(1 - p_{12}) + \frac{1}{c(p_1)} + \frac{1}{c(p_2)} \Big)
  \geq 0.
\end{equation}
If we set $c_1$ or $c_2$ to $0$ we quickly see that $(c(p_1) ,0)$ and $(0,c(p_2))$ are points on the boundary of $\cA$. If we set \eqref{eq:kequaltwo} to be zero we get the boundary between $(c(p_1) ,0)$ and $(0,c(p_2))$, which must be part of a conic section as it satisfies a quadratic equation. So it must be concave or convex. In particular it is enough to check whether the points on the line between $(c(p_1) ,0)$ and $(0,c(p_2))$, the line represented by $1 - c_1/c(p_1) - c_2/c(p_2) = 0$, are in $\cA$ or not. Because we are only considering points where $c_1$ and $c_2$ are positive, this line is contained in $\cA$ if and only if the second term of \eqref{eq:kequaltwo} is positive, that is, if $\ln(1 - p_{12}) + 1/c(p_1) + 1/c(p_2) \geq 0$, or, equivalently, $p_{12} \leq 1 - \sqrt{(1 - p_1)(1 - p_2)}$. Then and only then is $\cA$ convex.
In this case, the constant $c^\star$ defined as
\[
  c^\star = \max\big\{ |\bfm{c}| : \bfm{c} \in \conv(\cA) \cap \{\balpha t
  : t \in \mathbb{R}_{\geq 0} \} \big\}
\]
is given by a vector $\bfm{c^\star}$ which actually belongs to the set $\cA$
itself. In other words, independent $\bfm{t}$-sets with $t_1 = c^\star_1 (1 -
o(1))\ln n$ and $t_2 = c^\star_2 (1 - o(1))\ln n$ w.h.p.\ exist in $\Gnp$, and
a coloring can be found by greedily picking these sets as long as possible and
then coloring all remaining vertices with a new color.

On the other hand, if $p_{12} > 1 - \sqrt{(1 - p_1)(1 - p_2)}$, the situation is
quite different. In this case, the set $\cA$ is {\em concave} and the vector
$\bfm{c^\star}$ which determines the constant $c^\star$ does not belong to the
set $\cA$, but lies on its convex hull. In particular, it is given by
\[
  \bfm{c^\star} = \Big( \frac{\alpha_1}{\frac{\alpha_1}{c(p_1)} +
  \frac{\alpha_2}{c(p_2)}}, \frac{\alpha_2}{\frac{\alpha_1}{c(p_1)} +
  \frac{\alpha_2}{c(p_2)}} \Big).
\]
The optimal coloring is then achieved by looking at the ``extremal points''
$\bfm{c_1} = (c(p_1), 0)$ and $\bfm{c_2} = (0, c(p_2))$ and using independent
$(\bfm{c_1}(1 - \eps)\ln n)$-sets and $(\bfm{c_2}(1 - \eps)\ln n)$-sets as
color classes. Perhaps unsurprisingly, it is shown in
Proposition~\ref{prop:chrom-union} below that w.h.p.\ the chromatic number of
$\Gnp$ is then and only then the sum of the chromatic number of the two parts
$G[V_1]$ and $G[V_2]$, that is
\[
  \chi(\Gnp) = \big(1 + o(1)\big) \big(\chi(G(\alpha_1 n, p_1)) +
  \chi(G(\alpha_2 n, p_2)) \big) = \big(1 + o(1)\big) \frac{\alpha_1c(p_2) +
  \alpha_2c(p_1)}{c(p_1)c(p_2)} \frac{n}{\ln n}.
\]
For $p_{12} = 1 - \sqrt{(1 - p_1)(1 - p_2)}$ we have that $\cA$ is both convex
and concave since it is limited by a line---so any convex combination of
$\bfm{c}$-sets along this line yields a correct chromatic number asymptotically.

The shape of the set $\cA$ for fixed $0 < p_1 \leq p_2 < 1$ and depending on
$p_{12} \in (0, 1)$ is depicted on Figure~\ref{fig:figure-A} below.

\begin{figure}[!htbp]
  \captionsetup[subfigure]{labelformat=empty, textfont=scriptsize}
  \centering
  \begin{subfigure}{.5\textwidth}
    \centering
    \includegraphics[scale=0.7]{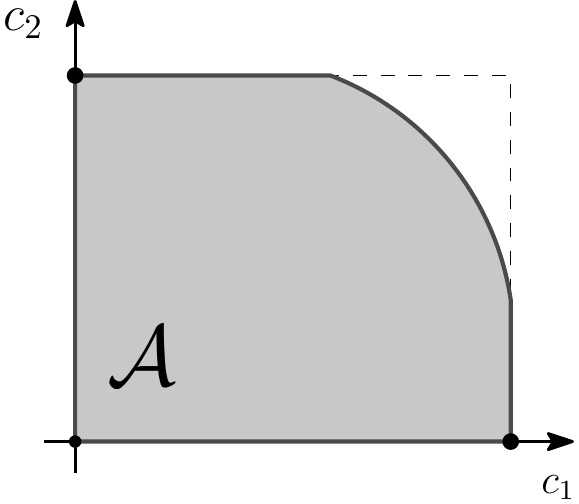}
    \caption{$0 < p_{12} < 1 - \sqrt{1 - p_1}$}
    \label{fig:sub-A}
  \end{subfigure}%
  \hfill
  \begin{subfigure}{.5\textwidth}
    \centering
    \includegraphics[scale=0.7]{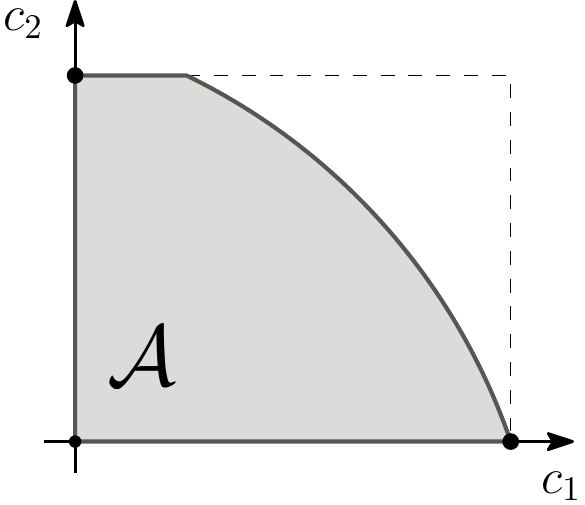}
    \caption{$1 - \sqrt{1 - p_1} \le p_{12} < 1 - \sqrt{1 - p_2}$}
    \label{fig:sub-B}
  \end{subfigure}%
  \vspace{2em}
  \begin{subfigure}{.34\textwidth}
    \centering
    \includegraphics[scale=0.7]{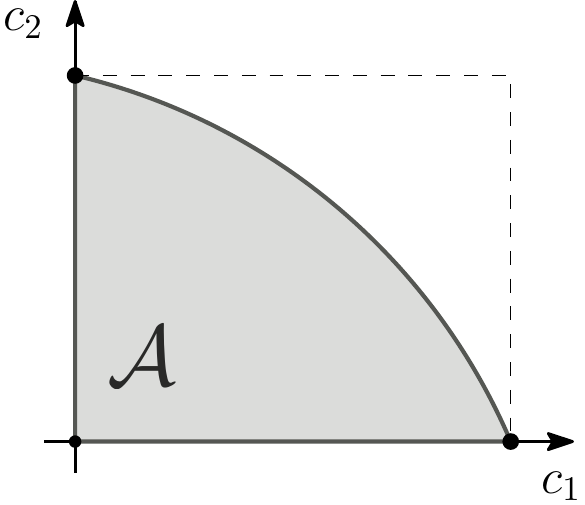}
    \caption{$1 - \sqrt{1 - p_2} \le p_{12} < 1 - \sqrt{(1 - p_1)(1 - p_2)}$}
    \label{fig:sub-C}
  \end{subfigure}%
  \hfill
  \begin{subfigure}{.33\textwidth}
    \centering
    \includegraphics[scale=0.7]{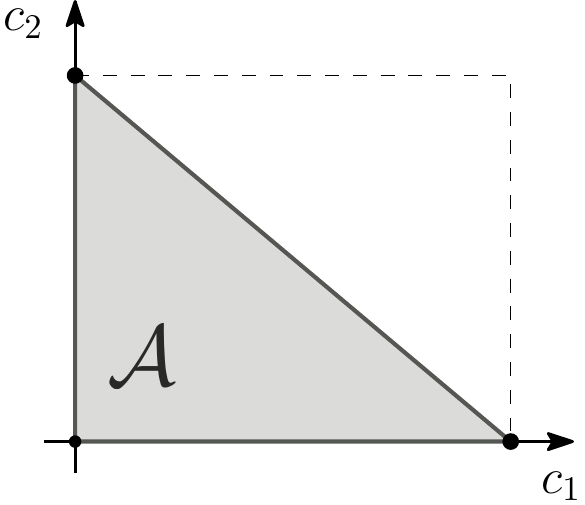}
    \caption{$p_{12} = 1 - \sqrt{(1 - p_1)(1 - p_2)}$}
    \label{fig:sub-D}
  \end{subfigure}%
  \hfill
  \begin{subfigure}{.33\textwidth}
    \centering
    \includegraphics[scale=0.7]{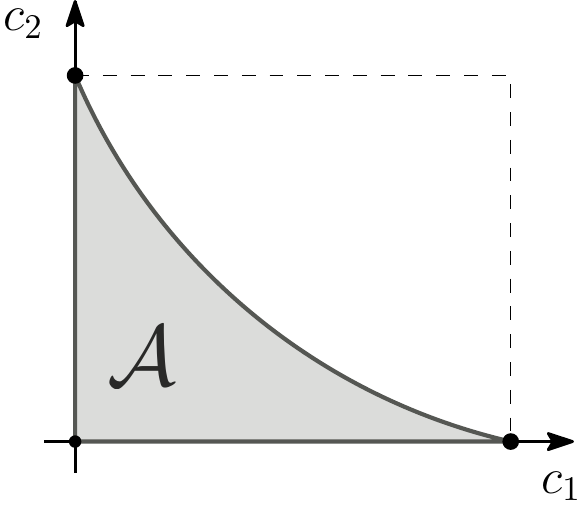}
    \caption{$1 - \sqrt{(1 - p_1)(1 - p_2)} < p_{12} < 1$}
    \label{fig:sub-E}
  \end{subfigure}
  \caption{Possibilities for $\cA$ in case $k = 2$, assuming $p_1 \leq p_2$ and
  then varying $p_{12}$. The dashed lines correspond to $c(p_1)$ and $c(p_2)$,
  that is, $c_1 = \frac{2}{-\ln(1 - p_1)}$ and $c_2 = \frac{2}{-\ln(1 - p_2)}$.}
  \label{fig:figure-A}
\end{figure}

Clearly, the constant $c^\star$ and the vector $\bfm{c^\star}$ that defines it
do not only depend on the set $\cA$ but also on the vector $\balpha$. In
Figure~\ref{fig:figure-alpha} we show how the vector $\bfm{c^\star}$ is defined.

\begin{figure}[!htbp]
  \captionsetup[subfigure]{labelformat=empty, textfont=scriptsize}
  \centering
  \begin{subfigure}{.33\textwidth}
    \centering
    \includegraphics[scale=0.7]{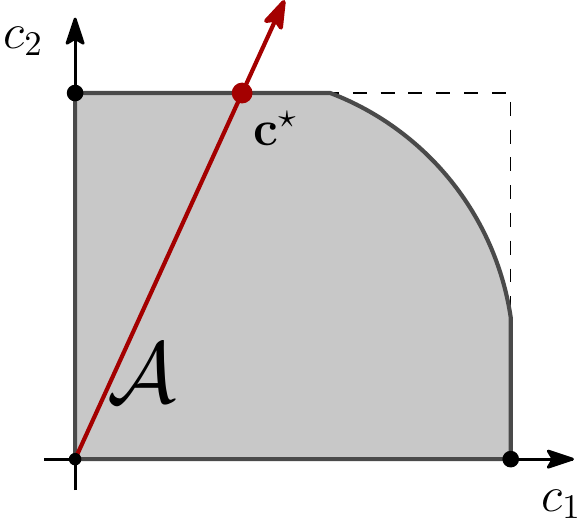}
  \end{subfigure}%
  \hfill
  \begin{subfigure}{.33\textwidth}
    \centering
    \includegraphics[scale=0.7]{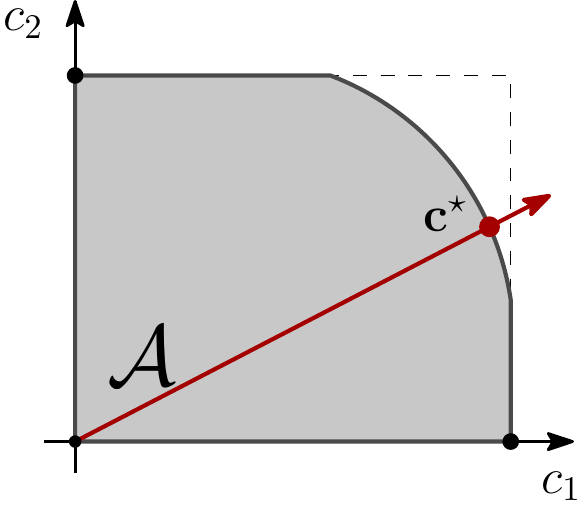}
  \end{subfigure}%
  \hfill
  \begin{subfigure}{.33\textwidth}
    \centering
    \includegraphics[scale=0.7]{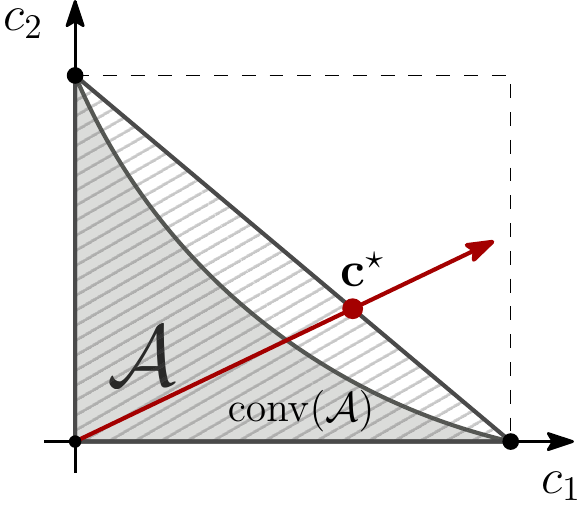}
  \end{subfigure}
  \caption{The red line represents the vector $\balpha$. The red point
  represents its intersection with $\conv(\cA)$, i.e.\ the vector
  $\bfm{c^\star}$.}
  \label{fig:figure-alpha}
\end{figure}

Worth noting is that if $p_{12} < 1 - \sqrt{(1 - p_1)}$ or $p_{12} < 1 -
\sqrt{(1 - p_2)}$ the inequalities $c_i < c(p_i)$ become relevant for the shape of
$\cA$. What this means for the chromatic number is that if
\[
  p_{12} \le 1 - \sqrt{(1 - p_1)} \quad \text{ and } \quad \alpha_1 \le c(p_1)
  \big( \ln(1-p_{12}) - \tfrac{1}{2} \ln (1-p_2) \big),
\]
then the vertex set $V_1$ has become so sparse and small, that we can color it
for free. So the chromatic number of $\Gnp$ is asymptotically bounded by the
chromatic number of $G[V_2]$ and is therefore
\[
  \chi(\Gnp) = \big(1 + o(1)\big) \frac{\alpha_2n}{c(p_2)\ln n} .
\]
The equivalent holds if $p_{12} \le 1 - \sqrt{(1 - p_2)}$ and $\alpha_2 \le
c(p_2) \big( \ln(1-p_{12}) - \frac{1}{2} \ln(1-p_1) \big)$ for the vertex set $V_2$.

\subsection{Concave set and the union of random graphs}

In this subsection we further explore how the random block graph $\Gnp$ behaves
in the general case $k \geq 2$ and when
\[
  p_{ij} > 1 - \sqrt{(1 - p_i)(1 - p_j)} \qquad \text{for all }\ 1 \leq i <
  j \leq k.
\]
In other words, when all of the bipartite graphs between the parts $G[V_i, V_j]$
are significantly denser than the densest graph $G[V_i]$. In case $k = 2$, this
is depicted on the rightmost parts of Figure~\ref{fig:figure-A} and
Figure~\ref{fig:figure-alpha}.

\begin{proposition}\label{prop:chrom-union}
  With high probability
  \[
    \chi(\Gnp) = \big(1 + o(1)\big) \Big( \sum_{1 \le i \le k} \chi(G(\alpha_i
    n, p_i)) \Big)
  \]
  if and only if $p_{ij} \ge 1-\sqrt{(1-p_i)(1-p_j)}$ for all $1 \leq i < j \leq
  k$.
\end{proposition}
\begin{proof}
  We first show that $p_{ij} \ge 1-\sqrt{(1-p_i)(1-p_j)}$ for all $1 \leq i < j \leq
  k$ implies the desired bound on the chromatic number. Recall, for a vector
  $\bfm{c} \in \R^k$ and $I \subseteq [k]$, the function $g(\bfm{c}, I)$ is
  defined as
  \[
    g(\bfm{c}, I) = \sum_{i \in I} c_i + \frac{1}{2} \sum_{i, j \in I} c_ic_j
    \ln(1 - p_{ij}).
  \]
  Note that we can reformulate this as
  \[
    g(\bfm{c}, I) = \Big(\sum_{i \in I} c_i \Big) \Big(1 - \sum_{i \in I}
    \frac{c_i}{c(p_i)}\Big) + \sum_{i \ne j \in I} c_i c_j \Big( \ln (1 -
    p_{ij}) + \frac{1}{c(p_i)} + \frac{1}{c(p_j)} \Big).
  \]
  By assumption of $p_{ij} \ge 1 - \sqrt{(1 - p_i)(1 - p_j)}$ we have that the
  term $\ln (1 - p_{ij}) + 1/c(p_i) + 1/c(p_j)$ is negative or zero for
  all $i, j \in [k]$. Consequently, $\cA \subseteq \cB \coloneqq \{\bfm{c} \in
  \R_{\ge 0}^k : 1 - \sum_{i \in [k]} \frac{c_i}{c(p_i)} \ge 0 \}$ and $\cB$ has
  as boundary a hyperplane and therefore is a convex set.

  For every $i \in [k]$, let
  \begin{equation}\label{eq:h-definition}
    \bfm{t_i} = c(p_i) \cdot \bfm{e}_i, \qquad h = \sum_{1 \leq j \leq k}
    \frac{\alpha_j}{|\bfm{t_j}|}, \qquad \text{and} \qquad \lambda_i =
    \frac{\alpha_i}{h |\bfm{t_i}|}
  \end{equation}
  Note that each $\lambda_i \in [0, 1]$ and $\sum_{1 \leq i \leq k} \lambda_i =
  1$. We claim that we can represent $\bfm{c^\star}$ as a convex combination of
  $\bfm{t_i}$'s like
  \[
    \bfm{c^\star} = \sum_{1 \leq i \leq k} \lambda_i \bfm{t_i}.
  \]
  Observe that, by definition
  \begin{equation}\label{eq:c-star-norm}
    c^\star = |\bfm{c^\star}| = \sum_{1 \leq i \leq k} \lambda_i |\bfm{t_i}| =
    \sum_{1 \leq i \leq k} \frac{\alpha_i}{h} = \frac{1}{h}.
  \end{equation}
 Each $\bfm{t_i} \in \cA \subseteq \cB$ is the intersection point of the
  hyperplane of $\cB$ with the corresponding axis. So, in fact, $\cA \subseteq
  \conv(\cA) = \cB$ and since $\bfm{c^\star}$ is a linear combination of the
  $\bfm{t_i}$'s it must lie on the boundary of $\cB$. That gives the upper bound
  on $\bfm{c^\star}$ and $\bfm{c^\star} \in \cB$ gives the lower bound. The rest
  now follows from the same strategy as in Theorem~\ref{thm:main-chromatic-num}
  and the fact that the number of different colors used is at most
  \[
    \sum_{1 \leq i \leq k} \frac{\alpha_i n}{(1 - \eps) |\bfm{t_i}| \ln n} +
    o\Big(\frac{n}{\ln n} \Big) \osref{\eqref{eq:h-definition}}= h \cdot
    \frac{n}{(1 - \eps)\ln n} + o \Big( \frac{n}{\ln n} \Big)
    \osref{\eqref{eq:c-star-norm}}= \frac{n}{c^\star \ln n} + o \Big(
    \frac{n}{\ln n} \Big).
  \]
  As for the other direction, whenever there is a $p_{ij} < 1 - \sqrt{(1 -
  p_i)(1 - p_j)}$ for a fixed $i \ne j \in [k]$ we can color $\Gnp$ in the
  following way. For every $h \in [k] \setminus \{i, j\}$ color each $V_h$
  separately with $\chi(G(\alpha_h n_h, p_h))$ colors. Then look at the graph
  induced by $V_i \cup V_j$. Clearly, $G[V_i \cup V_j]$ is distributed as the
  block graph $G(\alpha_i n_i + \alpha_j n_j, \balpha', P')$, where $\balpha' =
  (\alpha_i, \alpha_j)$ and
  $P' = \big(
    \begin{smallmatrix}
      p_{ii} & p_{ij} \\ p_{ji} & p_{jj}
    \end{smallmatrix}
  \big)$.
  Our observations from analysing the two-block case in
  Section~\ref{sec:twoblock} tell us that we can w.h.p.\ color this graph with
  asymptotically less colors than the sum of the chromatic numbers of the parts
  thus proving the proposition.
\end{proof}

\subsection{Convex set with homogeneous balanced partition}

The case where $\cA$ is convex can quickly turn out to be quite complicated.
Perhaps one of the cases worth mentioning is when the probability matrix $P$
contains only two different values, one for the diagonal and one for the
off-diagonal, and additionally $|V_1| = \dotsb = |V_k| = n/k$. So, for $P$ we
have
\[
  p_{ii} = p \quad \forall i \in [k] \qquad \text{and} \qquad p_{ij} = q \quad
  \forall i \ne j \in [k],
\]
with $p \ge q$. Then $\cA$ takes a convex shape since all the equations form
convex sets and the vector $\bfm{c^\star}$ must on the boundary of $\cA$. In
other words, $\bfm{c^\star} = (\frac{c^\star}{k}, \dotsc, \frac{c^\star}{k})$
where $c^\star = |\bfm{c^\star}|$ and, in particular, it must hold that
\[
  \sum_{i \in [k]} \frac{c^\star}{k} + \frac{1}{2} \sum_{i \in [k]} \Big(
  \frac{c^\star}{k} \Big)^2 \ln(1 - p) + \sum_{1 \leq i < j \leq k} \Big(
  \frac{c^\star}{k} \Big)^2 \ln(1 - q) \leq 0.
\]
By rearranging we get that $c^\star \leq -2 / \big( \frac{1}{k}\ln(1 - p) +
\frac{k-1}{k} \ln(1 - q) \big)$. Therefore, in case the previous is satisfied
with an equality, since $p \ge q$ and due to $\alpha_i = 1/k$, the equations
$g(\cdot, I)$ are automatically satisfied for all subsets of the indices $I
\subseteq [k]$, and hence $\bfm{c^\star}$ is maximal and in $\cA$. Applying
Theorem \ref{thm:main-chromatic-num}, we get
\[
  \chi(\Gnp) = \big(1 + o(1)\big) \frac{n}{2\ln n} \Big( -\frac{1}{k} \ln(1 -
  p) - \frac{k - 1}{k} \ln(1 - q) \Big).
\]

\bibliographystyle{abbrv}
\bibliography{references}

\end{document}